\documentclass[12pt]{article}
\title{Dp-finite and Noetherian NIP integral domains}
\author{Will Johnson}

\usepackage{amsmath, amssymb, amsthm}    	% need for subequations
\usepackage{fullpage} 	% without this, will have wide math-article-paper margins
\usepackage{amscd}
\usepackage{hyperref}
\usepackage[all]{xy}
\usepackage[T1]{fontenc}
\usepackage{lmodern}
\usepackage{centernot}
\usepackage{enumitem}
\usepackage{titlefoot}

\DeclareMathOperator*{\ind}{\raise0.2ex\hbox{\ooalign{\hidewidth$\vert$\hidewidth\cr\raise-0.9ex\hbox{$\smile$}}}}

\newcommand{\NIP}{\mathrm{NIP}}

\newcommand{\df}{\mathrm{def}}

\newcommand{\Ann}{\operatorname{Ann}}
\newcommand{\Frac}{\operatorname{Frac}}

\newcommand{\br}{\operatorname{br}}
\newcommand{\Sh}{\mathrm{Sh}}

\newcommand{\characteristic}{\operatorname{char}}
\newcommand{\Spec}{\operatorname{Spec}}

\newcommand{\dpr}{\operatorname{dp-rk}}

\newcommand{\Sub}{\operatorname{Sub}}

\newtheorem{theorem}{Theorem}[section] % numbered like the section

\newtheorem{lemma}[theorem]{Lemma}

\newtheorem{corollary}[theorem]{Corollary}
\newtheorem{fact}[theorem]{Fact}

\newtheorem{conjecture}[theorem]{Conjecture}

\newtheorem{proposition}[theorem]{Proposition}
\newtheorem{proposition-eh}[theorem]{Proposition(?)}
\newtheorem*{theorem-star}{Theorem}
\newtheorem*{conjecture-star}{Conjecture}
\newtheorem*{lemma-star}{Lemma}

\theoremstyle{definition}
\newtheorem{definition}[theorem]{Definition}
\newtheorem{example}[theorem]{Example}

\newtheorem{remark}[theorem]{Remark}

\newtheorem{step}{Step}

\theoremstyle{remark}

\newtheorem*{acknowledgment}{Acknowledgments}

\newcommand{\Qq}{\mathbb{Q}}
\newcommand{\Qp}{\mathbb{Q}_p}

\newcommand{\Zz}{\mathbb{Z}}

\newcommand{\Nn}{\mathbb{N}}

\newcommand{\Ff}{\mathbb{F}}

\newcommand{\Oo}{\mathcal{O}}
\newcommand{\mm}{\mathfrak{m}}
\newcommand{\pp}{\mathfrak{p}}
\newcommand{\qq}{\mathfrak{q}}

\newenvironment{stepproof}[1][\proofname]
               {
                 \proof[#1]
                 
               }
               {
                 \endproof
               }

\begin{document}

\maketitle\unmarkedfntext{

  }

\begin{abstract}
  We prove some results on NIP integral domains, especially those that
  are Noetherian or have finite dp-rank.  If $R$ is an NIP Noetherian
  domain that is not a field, then $R$ is a semilocal ring of Krull
  dimension 1, and the fraction field of $R$ has characteristic 0.
  Assuming the henselianity conjecture (on NIP valued fields), $R$ is
  a henselian local ring.  Additionally, we show that integral domains
  of finite dp-rank are henselian local rings.  Finally, we lay some
  groundwork for the study of Noetherian domains of finite dp-rank,
  and we classify dp-minimal Noetherian domains.
\end{abstract}

\section{Introduction}
In this paper, rings are commutative and unital.  We consider rings
and fields as structures in a language $\mathcal{L}$ expanding the
language of rings.  ``Definable'' means ``definable with parameters''.  Recall that a \emph{local ring} is a ring $R$ with a unique maximal
ideal $\mm$, and a local ring $R$ is \emph{henselian} if any
polynomial $c_nx^n + \cdots + c_1x + c_0 \in R[x]$ with $c_0 \in \mm$
and $c_1 \notin \mm$ has a root in $\mm$.  A \emph{valuation ring} is an integral domain $\Oo$ such that for every non-zero $x$ in the fraction field $\Frac(\Oo)$, at least one of $x$ or $x^{-1}$ is in $\Oo$.  Valuation rings are always local rings.
\subsection{Henselianity}
The following \emph{henselianity conjecture} is part of the
conjectural classification of NIP fields discussed in \cite{NIP-char, halevi-hasson-jahnke, hhj-v-top}.
\begin{conjecture}[Henselianity conjecture] \label{hens}
  If $\Oo$ is an NIP valuation ring, then $\Oo$ is henselian.
\end{conjecture}
\noindent Conjecture~\ref{hens} is known to hold in the following cases:
\begin{enumerate}
\item $\Frac(\Oo)$ has characteristic $p > 0$, or equivalently, $\Oo$
  is an $\Ff_p$-algebra \cite[Theorem~2.8]{prdf1a}.
\item $\Oo$ is dp-finite \cite[Corollary~4.16(3)]{prdf6}, meaning that
  the dp-rank of $\Oo$ is finite.
\end{enumerate}
See \cite[Section~4.2]{NIPguide} for information on dp-rank.  Note
that ``NIP'' is equivalent to ``bounded dp-rank'', and ``dp-finite''
is a proper subclass of ``NIP''.  \emph{Dp-minimal} means ``dp-rank 1''.  By \cite[Proposition~2.8(2)]{halevi-delbee}, a valuation ring $\Oo$ has the same
dp-rank as the corresponding valued field, i.e., the structure
$(K,+,\cdot,\Oo)$ where $K = \Frac(\Oo)$.

We
propose the following generalization of Conjecture~\ref{hens}.
\begin{conjecture}[Generalized henselianity conjecture] \label{ghens}
  If $R$ is an NIP ring, then $R$ is a direct product of finitely many
  henselian local rings.  In particular, if $R$ is an NIP integral
  domain, then $R$ is a henselian local ring.
\end{conjecture}
See Proposition~\ref{gh-reform} for some equivalent forms of this
conjecture.  When $R$ is an $\Ff_p$-algebra, Conjecture~\ref{ghens}
holds by \cite[Theorem~3.21]{fpcase}.  In this paper, we verify the
dp-finite case:
\begin{theorem}[{= Theorem~\ref{dft}}] \label{mt1}
  If $R$ is a dp-finite ring, then $R$ satisfies
  Conjecture~\ref{ghens}: $R$ is a direct product of finitely many
  henselian local rings.
\end{theorem}
We also show that the henselianity conjecture implies
Conjecture~\ref{ghens} in certain cases, namely Theorems~\ref{mt2} and
\ref{mt3} below.
\begin{theorem}[{= Theorem~\ref{xyz}}] \label{mt2}
  Assume the henselianity conjecture.  If $R$ is a Noetherian
  NIP ring, then $R$ satisfies Conjecture~\ref{ghens}: $R$ is a direct
  product of finitely many henselian local rings.
\end{theorem}
We will say more about NIP Noetherian rings in Section~\ref{intro2}
below.  Underlying Theorems~\ref{mt1} and \ref{mt2} is a result on
``$W_n$-rings.''
\begin{definition} \label{wn}
  Fix $n \ge 1$.  A ring $R$ is a \emph{$W_n$-ring} if the following
  condition holds: if $S$ is a finite subset of $R$, then there is a
  subset $S' \subseteq S$ with $|S'| \le n$ such that $S'$ and $S$
  generate the same ideal.  A \emph{$W_n$-domain} is a $W_n$-ring that
  is an integral domain.
\end{definition}
For example, a $W_1$-domain is the same thing as a valuation ring.  We
give several equivalent characterizations of $W_n$-rings in
Section~\ref{sec:breadth}.  Note that our terminology is slightly
different from \cite[Definition~2.5]{prdf5}, where ``$W_n$-ring''
meant ``$W_n$-domain''.\footnote{As an example of the difference, $\Zz/4\Zz$ is a $W_1$-ring but not a $W_1$-domain.}
\begin{theorem}[{= Theorem~\ref{nip-w}}] \label{mt3}
  Assume the henselianity conjecture.  If $R$ is an NIP $W_n$-ring,
  then $R$ satisfies Conjecture~\ref{ghens}: $R$ is a direct product
  of finitely many henselian local rings.
\end{theorem}
The connection between dp-finiteness, Noetherianity, and $W_n$-rings
is given as follows:
\begin{lemma}[{$\subseteq$ Corollary~\ref{hah} $\cup$ Lemma~\ref{dprbr}}]
  Let $R$ be an NIP ring.
  \begin{enumerate}
  \item If $R$ is Noetherian, then $R$ is a $W_n$-ring for some $n$.
  \item If $R$ is dp-finite and $R/\mm$ is infinite for every maximal
    ideal $\mm \subseteq R$, then $R$ is a $W_n$-ring for $n = \dpr(R)$.
  \end{enumerate}
\end{lemma}
\begin{remark}
  Theorems~\ref{mt1}, \ref{mt2}, and \ref{mt3} suggest the following
  natural question: does the henselianity conjecture (Conjecture~\ref{hens}) imply Conjecture~\ref{ghens}?  I could not see how to prove
  this.  A more promising question is whether the \emph{Shelah
  conjecture} implies Conjecture~\ref{ghens}.  The Shelah conjecture
  says that if $K$ is an NIP field, then one of the following holds:
  $K$ is finite, $K$ is algebraically closed, $K$ is real closed, or
  $K$ admits a non-trivial henselian valuation.  Like the henselianity
  conjecture, the Shelah conjecture is part of the conjectural
  classification of NIP fields.  In fact, the Shelah conjecture
  implies a complete classification of NIP fields
  \cite[Theorem~7.1]{NIP-char}.  By work of Halevi, Hasson, and Jahnke
  \cite{hhj-v-top}, the Shelah conjecture implies the henselianity
  conjecture.  Perhaps similar arguments could be used to show that
  the Shelah conjecture implies Conjecture~\ref{ghens}.
  Proposition~\ref{gh-reform} gives some alternate formulations of
  Conjecture~\ref{ghens} which may be useful for this purpose.
\end{remark}
\begin{remark}
  Conjecture~\ref{ghens} is a statement about NIP rings, but it also
  has applications for definable field topologies on NIP fields (see Section~\ref{sec-moved-8.1-start} for the definition of ``definable topology'').  In a
  later paper~\cite{tops-rings}, we will show that
  Conjecture~\ref{ghens} implies the following statement:
  \begin{quote}
    Let $(K,+,\cdot,\ldots)$ be an NIP expansion of a field, and
    $\tau$ be a definable field topology on $K$.  Then $K$ is large in
    the sense of Pop \cite{Pop-little} and $\tau$ is ``generalized
    topologically henselian'' in the sense of
    \cite[Definition~8.1]{hensquot2} (essentially meaning that the
    inverse function theorem holds for polynomial maps).
  \end{quote}
  In particular, this statement holds when
  $\characteristic(K) > 0$ or $K$ is dp-finite.
\end{remark}

\subsection{NIP Noetherian domains} \label{intro2}
Another goal of this paper is to initiate the study of NIP Noetherian
domains.  If we want a classification of NIP Noetherian domains, we
should either work modulo the conjectural classification of NIP fields
(see \cite[Theorem~7.1]{NIP-char}), or focus on cases of low dp-rank
such as dp-minimal and dp-finite rings.

There are several motivations for studying NIP Noetherian domains.
For model theorists, the problem of classifying NIP Noetherian domains
can be seen as a first step beyond classifying NIP fields, towards the
harder problem of classifying NIP integral domains.  Additionally, a
classification of NIP Noetherian domains may provide new examples of
NIP theories.  For commutative algebraists, the class of Noetherian
domains has central importance and one would like to know which
Noetherian domains are amenable to the tools of model theory.

In this paper, we take some first steps towards the classification of
NIP Noetherian domains, especially dp-finite Noetherian domains.  Here
are the main results.  Recall that the \emph{Krull dimension} of a ring $R$ is the supremum of lengths of finite chains of prime ideals
$\pp_0 \subsetneq \pp_1 \subsetneq \cdots \subsetneq \pp_n$ in $R$.
\begin{theorem}[{$\subseteq$ Theorem~\ref{charzero} $\cup$ Corollary~\ref{hah}}]
  Let $R$ be an NIP Noetherian domain.  Suppose $R$ is not a field.
  \begin{enumerate}
  \item $\Frac(R)$ has characteristic 0.
  \item $R$ has finitely many maximal ideals $\mm_1,\ldots,\mm_n$.
  \item $R$ has Krull dimension 1, and so $\{0,\mm_1,\ldots,\mm_n\}$
    are the only prime ideals of $R$.
  \end{enumerate}
\end{theorem}
In the dp-finite case, we get the following trichotomy:
\begin{theorem}[{= Theorem~\ref{tricho1}}]
  Let $R$ be a dp-finite Noetherian domain.  Then $R$ is a local ring
  and one of three things happens:
  \begin{enumerate}
  \item $R$ is not a field.   $\Frac(R)$ and the residue field both have characteristic 0.
  \item $R$ is not a field.  $\Frac(R)$ has characteristic 0 and the
    residue field is finite.
  \item $R$ is a field.
  \end{enumerate}
\end{theorem}
Lastly, we get a classification of dp-minimal Noetherian domains,
building off the known classifications of dp-minimal fields and
dp-minimal valuation rings.  Recall that a structure is \emph{dp-minimal} if it has dp-rank 1.  See \cite[Theorem~1.3]{dpm2} for the
classification of dp-minimal fields, which is too complicated to state
here.  Recall that a \emph{discrete valuation ring} (DVR) is a valuation ring $\Oo$ with fraction field $K$ such that the corresponding value group $K^\times/\Oo^\times$ is isomorphic to $\Zz$.  By \cite[Theorems~1.5--1.6]{dpm2}, dp-minimal DVRs are classified as follows:
\footnote{Recall that a valuation ring $(\Oo,+,\cdot)$ has the same dp-rank as its corresponding valued field $(\Frac(\Oo),+,\cdot,\Oo)$ by \cite[Proposition~2.8(2)]{halevi-delbee}, so the reference to \cite[Theorems~1.5--1.6]{dpm2} is relevant.}
\begin{fact}\label{dvr-class}
  Let $R$ be a DVR.  Then $R$ is dp-minimal iff $R$ is henselian,
  $\Frac(R)$ has characteristic 0, and one of the following holds:
  \begin{enumerate}
  \item The residue field is a dp-minimal field of characteristic 0.
  \item The residue field is finite.
  \end{enumerate}
\end{fact}
We call these two cases the ``equicharacteristic'' and ``mixed
characteristic'' cases of dp-minimal DVRs.  If $R$ is an
equicharacteristic dp-minimal DVR, then $R \equiv K[[t]]$ for some
dp-minimal field $K$ of characteristic 0, by the Ax-Kochen-Ershov principle.  If $R$ is a mixed
characteristic dp-minimal DVR, then $R \equiv \Oo_K$, where $K$ is a
non-archimedean local field of characteristic 0 (i.e., a finite
extension of $\Qq_p$), and $\Oo_K$ is its ring of integers.\footnote{Indeed, the valued field $\Frac(R)$ is henselian with finite residue field and value group $\Zz$, so it is $p$-adically closed by \cite[Theorem~3.1]{Prestel-roquette}, hence elementarily equivalent to an extension of $\Qq_p$ by \cite[Fact~3.7]{halevi-hasson-jahnke}.}
\begin{theorem}[{= Theorem~\ref{class}}] \label{1-class}
  The following is a complete list of dp-minimal Noetherian domains:
  \begin{enumerate}
  \item Equicharacteristic dp-minimal DVRs.
  \item Finite index subrings of mixed characteristic dp-minimal DVRs.
  \item Dp-minimal fields.
  \end{enumerate}
  Here, a ``finite index'' subring of $R$ is a subring $R_0 \subseteq
  R$ with finite index in $(R,+)$.
\end{theorem}
In subsequent work with d'Elb\'ee and Halevi~\cite{dEHJ},
Theorem~\ref{1-class} is generalized to a complete classification of
dp-minimal integral domains.

\section{Breadth in modular lattices and modules} \label{sec:breadth}
In this section, we review some facts about breadth of lattices,
modules, and rings, from \cite{prdf1b, prdf3, prdf4, prdf5}.
\begin{remark}
  Breadth was called ``reduced rank'' in \cite{prdf3,prdf4} and earlier
  drafts of \cite{prdf1b} and called ``cube rank'' or ``weight'' in
  \cite{prdf5}.  Yatir Halevi pointed out that ``breadth'' is
  the standard terminology in lattice theory \cite[Exercise~II.5.6]{birkhoff}.  It seems reasonable to
  extend this terminology from lattices to modules and rings.
\end{remark}
In the following, a ``lattice'' means an \emph{unbounded} lattice,
i.e., a partial order $(\Lambda,\le)$ in which every finite
\emph{non-empty} subset has a infimum and supremum.  Unbounded
lattices can be regarded as algebraic structures
$(\Lambda,\wedge,\vee)$ satisfying certain identities~\cite[Section~I.5]{birkhoff}, where $\wedge$
and $\vee$ are the binary infimum and supremum, respectively.  In particular, a ``sublattice'' means a
subset $\Lambda_0 \subseteq \Lambda$ closed under $\wedge$ and $\vee$,
but not necessarily containing $\max(\Lambda)$ or $\min(\Lambda)$ when
they exist.

Recall that a lattice $\Lambda$ is \emph{modular}
\cite[Section~I.7]{birkhoff} if it satisfies the modular identity:
\begin{equation*}
  (x \vee a) \wedge b = (x \wedge b) \vee a \text{ when } a \le
  b. % \tag{$\ast$}
\end{equation*}
% The important thing is not the details of ($\ast$), but the fact that
For any $R$-module $M$, the lattice $\Sub_R(M)$ of $R$-submodules of $M$ is
modular \cite[Theorem VII.1]{birkhoff}.

Let $\Lambda$ be a non-empty modular lattice, such as the lattice
$\Sub_R(M)$.  The \emph{length} of $\Lambda$, written $\ell(\Lambda)$, is the largest $n$ such that there is a chain $x_0 < x_1 < \cdots < x_n$ in $\Lambda$, or $\infty$ if no maximum exists.  The \emph{breadth}
of $\Lambda$, written $\br(\Lambda)$, is an element of
$\{0,1,2,\ldots\} \cup \{\infty\}$ characterized in one of the
following ways:
\begin{enumerate}
\item $\br(\Lambda) \ge n$ if there is a sublattice of $\Lambda$
  isomorphic to the powerset of $n$.
\item $\br(\Lambda) \le n$ if for any $x_1,\ldots,x_{n+1} \in
  \Lambda$, there is some $i$ such that
  \begin{equation*}
    x_1 \vee \cdots \vee x_{n+1} = x_1 \vee \cdots \vee \widehat{x_i}
    \vee \cdots \vee x_{n+1},
  \end{equation*}
  where the hat indicates omission.
\item $\br(\Lambda) \le n$ if for any $x_1,\ldots,x_{n+1} \in
  \Lambda$, there is some $i$ such that
  \begin{equation*}
    x_1 \wedge \cdots \wedge x_{n+1} = x_1 \wedge \cdots \wedge \widehat{x_i}
    \wedge \cdots \wedge x_{n+1}.
  \end{equation*}
\end{enumerate}
These are equivalent by \cite[Lemma~9.9]{prdf1b}.  If $R$ is a ring
and $M$ is an $R$-module, the \emph{length} of $M$, written $\ell(M)$
or $\ell_R(M)$, is the length of the lattice $\Sub_R(M)$.
Equivalently, $\ell(M)$ is the largest $n$ such that there exists a
chain of submodules $M_0 \subsetneq M_1 \subsetneq \cdots \subsetneq
M_n \subseteq M$.  The \emph{breadth} of $M$, written $\br(M)$ or
$\br_R(M)$, is the breadth of the lattice $\Sub_R(M)$.
\begin{remark}\label{module-breadth}
By \cite[Remark~6.7]{prdf3} and
\cite[Proposition~7.3]{prdf4}, we can equivalently characterize breadth as follows:
\begin{enumerate}
\item $\br(M) \ge n$ iff there are submodules $N \le N' \le M$ such
  that $N'/N$ is a direct sum of $n$ non-trivial $R$-modules.
\item $\br(M) \le n$ iff for any $x_1,\ldots,x_{n+1} \in M$, there is
  $i$ such that
  \begin{equation*}
    Rx_1 + \cdots + Rx_{n+1} = Rx_1 + \cdots + \widehat{Rx_i} + \cdots
    + Rx_{n+1}.
  \end{equation*}
\end{enumerate}  
\end{remark}
\begin{fact}[{\cite[Lemma~2.2]{prdf5}}] \label{base}
  If $M$ is an $R$-module and $R'$ is a subring of $R$, then $\br_R(M)
  \le \br_{R'}(M)$.
\end{fact}
\begin{remark} \label{base2}
  Similarly, if $I$ is an ideal in $R$ and $M$ is an $R/I$-module,
  regarded as an $R$-module in the natural way, then $\br_R(M) =
  \br_{R/I}(M)$, because the lattices of submodules are the same:
  \begin{equation*}
    \Sub_R(M) = \Sub_{R/I}(M).
  \end{equation*}
\end{remark}
\begin{remark} \label{vec-br}
  If $V$ is a vector space over a field $K$, then $\br_K(V) = \ell_K(V) =
  \dim_K(V)$.  This follows easily using the definitions, or using the following facts.
\end{remark}
\begin{fact}[{\cite[Proposition~6.9]{prdf3}}] \label{additive}
  Let $R$ be a ring and $0 \to M_1 \to M_2 \to M_3 \to 0$ be a short
  exact sequence of $R$-modules.  Then the following hold:
  \begin{enumerate}
  \item $\br(M_1) \le \br(M_2)$ and $\br(M_3) \le \br(M_2)$.
  \item $\br(M_2) \le \br(M_1) + \br(M_3)$.
  \item If the sequence splits (i.e., $M_2 \cong M_1 \oplus M_3$),
    then $\br(M_2) = \br(M_1) + \br(M_3)$.
  \end{enumerate}
\end{fact}
Fact~\ref{additive} is analogous to the classic Jordan-H\"older
theorem for length:
\begin{fact}[{\cite[Theorem~2.13(a)]{eisenbud}}] \label{JH}
  If $0 \to M_1 \to M_2 \to M_3 \to 0$ is a short exact sequence of
  $R$-modules, then $\ell(M_2) = \ell(M_1) + \ell(M_3)$.
\end{fact}
If $R$ is a ring, the breadth $\br(R)$ is the breadth of $R$ as an
$R$-module, i.e., the breadth of the lattice of ideals.  When $R$ is
an integral domain with fraction field $K$, we have $\br(R) =
\br_R(K)$ by \cite[Lemma~2.4]{prdf5}.  The $W_n$-rings of Definition~\ref{wn} are precisely the rings with $\br(R) \le n$, by Remark~\ref{module-breadth}.  Valuation rings are the same thing as $W_1$-domains
\cite[Proposition~2.6]{prdf5}.

A \emph{multivaluation ring} is a finite intersection of valuation
rings.  Multivaluation rings are examples of $W_n$-domains:
\begin{fact} \label{multival}
  Let $\Oo_1,\ldots,\Oo_n$ be pairwise incomparable non-trivial
  valuation rings on a field $K$, and let $R = \bigcap_{i = 1}^n
  \Oo_i$.
  \begin{enumerate}
  \item $R$ has exactly $n$ maximal ideals $\mm_1,\ldots,\mm_n$, and
    $\Oo_i$ is the localization of $R$ at $\mm_i$.
  \item $R$ has breadth $n$.
  \end{enumerate}
\end{fact}
The first point is well-known; see Lemma~3.2.6 and Theorem~3.2.7 in
\cite{PE}.  The second point is \cite[Lemma~6.5]{prdf3}.  The first
point is related to the approximation theorem for valuation topologies
\cite[Theorem~2.4.1]{PE}, though note we are not assuming the $\Oo_i$ are
independent.

\begin{fact}[{\cite[Proposition~2.12]{prdf5}}] \label{tilde}
  Let $R$ be a $W_n$-domain.  The integral closure $\tilde{R}$ is a
  multivaluation ring $\bigcap_{i = 1}^m \Oo_i$ with $m \le n$.
\end{fact}

The \emph{width} of a poset $(P,\le)$ is the maximum size of an
antichain.  Dilworth's theorem~\cite{dilworth} says that if $P$ has width $n <
\omega$, then $P$ is a union of $n$ chains.\footnote{Dilworth's
theorem is sometimes stated only for finite posets, but it extends to
infinite posets by propositional compactness or Tychonoff's theorem.}
If $R$ is a ring, let $\Spec R$ denote the poset of prime ideals of
$R$ ordered by inclusion.
\begin{remark} \label{width-rem}
  If $R$ is a ring and $\br(R) \le n$, then the width of $\Spec R$ is
  at most $n$.  Indeed, suppose $\pp_1,\ldots,\pp_{n+1} \in \Spec R$
  are incomparable.  The fact that $\pp_1 \not\supseteq \pp_i$ for $i
  > 1$ implies that $\pp_1 \not\supseteq \bigcap_{i=2}^{n+1} \pp_i$ by
  \cite[Exercise~7.4.11]{dummit-foote}, and so
  \begin{equation*}
    \pp_1 \cap \cdots \cap \pp_{n+1} \ne \pp_2 \cap \cdots \cap
    \pp_{n+1}.
  \end{equation*}
  Similarly,
  \begin{equation*}
    \pp_1 \cap \cdots \cap \pp_{n+1} \ne \pp_1 \cap \cdots \cap
    \widehat{\pp_i} \cap \cdots \cap \pp_{n+1}
  \end{equation*}
  for any $i$.  Therefore the lattice of ideals has breadth at least
  $n+1$, meaning that $\br(R) \ge n+1$.
\end{remark}
\begin{fact} \label{fin-wid}
  If $R$ is an NIP ring, then the width of $\Spec R$ is finite \cite[Proposition~2.1, Remark~2.2]{halevi-delbee}.
\end{fact}

\section{Henselizing classes}
In \cite{fpcase}, a certain argument was used to show that NIP
$\Ff_p$-algebras satisfy Conjecture~\ref{ghens}.  This same argument
can be used in other settings, replacing ``NIP $\Ff_p$-algebras'' with
other classes of NIP rings.  In this section we identify the abstract
conditions needed to run the argument.
% The FOUR uses are
% 1. Proposition~\ref{gh-reform} (all NIP rings)
% 2. Theorem~\ref{nip-w} (NIP rings of finite breadth)
% 3. Lemma~\ref{lem:dpf} (dp-finite rings of finite breadth)
% 4. Theorem~\ref{dft} (dp-finite rings)
Recall that a set $D \subseteq M^n$ in a structure $M$ is
\emph{externally definable} if $D = D' \cap M^n$ for some definable
set $D' \subseteq N^n$ in an elementary extension $N \succeq M$.
Equivalently, $D = \phi(M,b)$ for some formula $\phi(x,y)$ and
parameter $b$ coming from an elementary extension of $M$.  The
intuition is that $D$ is defined using parameters from outside $M$.  Any prime ideal in a ring is externally definable \cite[Fact~4.1]{dEHJ}.
\begin{remark} \label{shelah-rem}
  If $M$ is any structure, the \emph{Shelah expansion} $M^\Sh$ is the
  expansion of $M$ by all externally definable sets.  When $M$ is NIP,
  the definable sets in $M^\Sh$ are exactly the externally definable
  sets in $M$ \cite[Proposition~3.23]{NIPguide}.  From this, one can
  show that $M^\Sh$ is NIP \cite[Corollary~3.24]{NIPguide}.  A similar
  argument shows that $M^\Sh$ has the same dp-rank as $M$~\cite[Fact~3.1]{dEHJ}.  For
  example, if $M$ is dp-finite, then $M^\Sh$ is dp-finite.
\end{remark}
\begin{definition} \label{good}
  A class $\mathcal{K}$ of rings is \emph{pre-henselizing} if it is closed under
  the following conditions:
  \begin{description}
  \item[(Loc)] If $R \in \mathcal{K}$ and $R$ is an integral domain,
    then any localization $S^{-1}R$ is in $\mathcal{K}$.
  \item[(El)] If $R \in \mathcal{K}$ and $R' \equiv
    R$, then $R' \in \mathcal{K}$.
  \item[(Quot)] If $R \in \mathcal{K}$ and $I$ is an externally
    definable ideal in $R$, then $R/I \in \mathcal{K}$.
  \item[(Free)] Let $R$ be in $\mathcal{K}$, and let
    $S$ be an $R$-algebra that is free of finite rank as an
    $R$-module.  Then $S \in \mathcal{K}$.
%%   \item[(Prod)] If $R \in \mathcal{K}$ and $R = R_1 \times R_2$, then
%%     $R_1, R_2 \in \mathcal{K}$.
  \end{description}
\end{definition}
The following classes are pre-henselizing:
\begin{enumerate}
\item $\Ff_p$-algebras:  the five properties are clear.
\item NIP rings: (El) is well-known.  (Loc),
  (Quot), and (Free) hold because the resulting rings are interpretable in
  $R^{\Sh}$.  For localizations, this is \cite[Theorem~2.11]{fpcase}.
\item Dp-finite rings: Similar to NIP rings.
\end{enumerate}
Any intersection of pre-henselizing classes is pre-henselizing.  For example, the class of
NIP $\Ff_p$-algebras is pre-henselizing.
\begin{definition} \label{vgood}
  A \emph{problematic ring} is an integral domain $R$ with exactly two maximal
  ideals $\mm_1, \mm_2$, such that $R/\mm_i$ is infinite for $i =
  1,2$.  A \emph{henselizing} class of rings is a pre-henselizing class
  containing no problematic rings.
\end{definition}
Lemma~3.9 in \cite{fpcase} says that the class of NIP $\Ff_p$-algebras
has no problematic rings, and thus is henselizing.
\begin{proposition} \label{blackbox}
  Let $\mathcal{K}$ be a henselizing class of NIP rings.  Suppose $R \in
  \mathcal{K}$.  Then\ldots
  \begin{enumerate}
  \item $R$ is a finite product of henselian local rings.
  \end{enumerate}
  If additionally $R$ is an integral domain, then\ldots
  \begin{enumerate}[resume]
  \item $R$ is a henselian local ring.
  \item The prime ideals of $R$ are linearly ordered by inclusion.
  \end{enumerate}
\end{proposition}
\begin{proof}
  The proofs of Lemmas~3.12--3.13, Theorem~3.15, Corollary~3.16,
  Lemma~3.19, Proposition~3.20, and Theorems~3.21--3.22 in
  \cite{fpcase} apply with minimal changes, replacing ``NIP
  $\Ff_p$-algebra'' with ``member of $\mathcal{K}$''.  In more detail,
  if we fix the class $\mathcal{K}$, then the following things are
  true:
  \begin{step}\label{step1}
    Suppose $R \in \mathcal{K}$ is an integral domain, and $\pp_1,
    \pp_2$ are two prime ideals such that $R/\pp_1$ and $R/\pp_2$ are
    infinite.  Then $\pp_1$ and $\pp_2$ are comparable.
  \end{step}
  \begin{stepproof}
    Otherwise, let $S = R \setminus (\pp_1 \cup \pp_2)$ and let $R'$
    be $S^{-1}R$, which belongs to $\mathcal{K}$ by (Loc).  The ring
    $R'$ has exactly two maximal ideals $\pp_1R'$ and $\pp_2R'$, and
    the quotients $R'/\pp_iR'$ are infinite because of the embeddings
    $R/\pp_i \to R'/\pp_iR'$.  Thus $R'$ is problematic, contradicting
    the fact that $\mathcal{K}$ is henselizing.
  \end{stepproof}
  \begin{step} \label{step2}
    $\mathcal{K}$ does not contain a domain $R$ with exactly two
    maximal ideals.
  \end{step}
  \begin{stepproof}
    The proof of \cite[Lemma~3.13]{fpcase} applies, using (El) to make
    $R$ be saturated at the start of the proof.  Step~\ref{step1}
    replaces the use of \cite[Lemma~3.12]{fpcase}.
  \end{stepproof}
  \begin{step} \label{step3}
    If $R \in \mathcal{K}$ is an integral domain, then the prime
    ideals of $R$ are linearly ordered by inclusion.
  \end{step}
  \begin{stepproof}
    Otherwise, take two incomparable prime ideals $\pp_1, \pp_2$ in
    $R$, let $S = \pp_1 \cup \pp_2$, and let $R' = S^{-1}R$.  Then $R'
    \in \mathcal{K}$ by (Loc), and $R'$ has exactly two maximal
    ideals, contradicting Step~\ref{step2}.
  \end{stepproof}
  \begin{step}\label{step3.5}
    If $R \in \mathcal{K}$ and $\pp_1, \pp_2, \qq$ are prime ideals
    with $\pp_1, \pp_2 \supseteq \qq$, then $\pp_1$ and $\pp_2$ are
    comparable.
  \end{step}
  \begin{stepproof}
    Otherwise, $\pp_1/\qq$ and $\pp_2/\qq$ are incomparable prime
    ideals in the domain $R/\qq$.  But $\qq$ is externally definable
    in $R$ by \cite[Fact~4.1]{dEHJ}, so $R/\qq \in
    \mathcal{K}$ by (Quot).  This contradicts Step~\ref{step3}.
  \end{stepproof}
  Recall from the previous section that the width of a poset is the
  maximum size of an antichain in the poset.  As in
  \cite[Definition~3.17]{fpcase}, say that a poset $(P,\le)$ is a
  \emph{forest} if $\{x \in P : x \ge c\}$ is a chain for every $c \in
  P$.    Following~\cite[Definition~3.18]{fpcase}, say that a ring $R$ is ``good'' if the poset of prime ideals in $R$ is a forest of finite width.
  \begin{step}\label{step4}
    If $R \in \mathcal{K}$, then $R$ is good.
    %% the poset of prime ideals in $R$ is a
%%     forest of finite width.  (In the terminology of
%%     \cite[Definition~3.18]{fpcase}, $R$ is ``good''.)
  \end{step}
  \begin{stepproof}
    The poset of prime ideals has finite width by Fact~\ref{fin-wid}, because $R$ is NIP.\footnote{This is the only point in the proof where we use the assumption that the rings in $\mathcal{K}$ are NIP.}  It is a forest
    by Step~\ref{step3.5}.
  \end{stepproof}
  \begin{step}\label{step5}
    If $R \in \mathcal{K}$ is a local ring, then $R$ is henselian.
  \end{step}
  \begin{stepproof}
    The proof of \cite[Proposition~3.20]{fpcase} applies, using (Free)
    to see that $\mathcal{K}$ contains the ring
    $R[x_1,\ldots,x_n]/(P_1(x_1),\ldots,P_i(x_i))$ appearing in the
    proof, and using Step~\ref{step4} to conclude that this ring is
    ``good''.
  \end{stepproof}
  \begin{step}\label{step6}
    If $R \in \mathcal{K}$ is arbitrary, then $R$ is a finite
    product of henselian local rings.
  \end{step}
  \begin{stepproof}
    $R$ is ``good'' by Step~\ref{step4}, and therefore a finite
    product of local rings $R = \prod_{i = 1}^n R_i$ by
    \cite[Lemma~3.19(3)]{fpcase}.  The kernel of each projection $R
    \to R_i$ is definable---in fact, generated by an idempotent---and
    so $R_i \in \mathcal{K}$ by (Quot).  Then $R_i$ is henselian by
    Step~\ref{step5}.
  \end{stepproof}
  \begin{step} \label{step7}
    If $R \in \mathcal{K}$ is integral, then $R$ is a henselian local
    ring.
  \end{step}
  \begin{stepproof}
    Clear from Step~\ref{step6}, since $R$ is not a product of more
    than one ring.
  \end{stepproof}
  \noindent This completes the proof of Proposition~\ref{blackbox}.
\end{proof}
As a first application, we mildly strengthen the main theorem of
\cite{fpcase}.  Recall that the characteristic $\characteristic(R)$ of
a ring $R$ is the unique $n \in \Nn$ such that $n\Zz$ is the kernel of
the unique homomorphism $\Zz \to R$.  For example, $\Zz/4\Zz$ has
characteristic 4, though it is not an $\Ff_p$-algebra for any prime
$p$.
\begin{theorem}
  Let $R$ be an NIP ring of positive characteristic.  Then $R$ is a
  finite product of henselian local rings.
\end{theorem}
\begin{proof}
  Let $\mathcal{K}_{\NIP}$ be the class of NIP rings, which is
  pre-henselizing as noted above.  Let $\mathcal{K}_+$ be the class of
  rings of positive characteristic.  It is easy to see that
  $\mathcal{K}_+$ is pre-henselizing.  Then $\mathcal{K}_{\NIP} \cap
  \mathcal{K}_+$, the class of NIP rings of positive characteristic,
  is pre-henselizing.  If $R \in \mathcal{K}_{\NIP} \cap
  \mathcal{K}_+$ is problematic, then $R$ is an integral domain, hence
  $\characteristic(R) = p$ for some prime $p > 0$.  Then $R$ is an
  $\Ff_p$-algebra.  But as noted above, there are no problematic NIP
  $\Ff_p$-algebras, by \cite[Lemma~3.9]{fpcase}.  Therefore,
  $\mathcal{K}_{\NIP} \cap \mathcal{K}_+$ is henselizing, and
  Proposition~\ref{blackbox} applies.
\end{proof}
We can use the machinery of Proposition~\ref{blackbox} to give some
reformulations of Conjecture~\ref{ghens}, the generalized henselianity
conjecture:
\begin{proposition} \label{gh-reform}
  The following are equivalent:
  \begin{enumerate}
  \item Conjecture~\ref{ghens} holds: every NIP ring is a finite
    product of henselian local rings.
  \item Every NIP integral domain is a henselian local ring.
  \item Every NIP integral domain is a local ring.
  \item If $R$ is an NIP integral domain, then the prime ideals of $R$ are linearly ordered.
  \end{enumerate}
\end{proposition}
\begin{proof}
  $(1)\implies(2)$, $(2)\implies(3)$, and $(4)\implies(3)$ are clear.  Suppose (3) holds.
  Let $\mathcal{K}$ be the class of NIP rings.  As noted above,
  $\mathcal{K}$ is pre-henselizing.  By (3), no NIP integral domain has exactly
  two maximal ideals, which implies that $\mathcal{K}$ is henselizing
  (Definition~\ref{vgood}).  By Proposition~\ref{blackbox}, every ring
  in $\mathcal{K}$ is a finite product of henselian local rings, which
  is (1).  By Proposition~\ref{blackbox} again, the prime ideals in any integral domain in $\mathcal{K}$ are linearly ordered, which is (4).
\end{proof}
\begin{remark}
  The generalized henselianity conjecture is meant to be a
  generalization of Conjecture~\ref{hens}, the henselianity conjecture
  for NIP valuation rings.  Arguably, a more direct analogue of
  Conjecture~\ref{hens} would be the statements:
\begin{enumerate}
  \setcounter{enumi}{4}
\item Every NIP local ring is henselian.
\item Every NIP local integral domain is henselian.
\end{enumerate}
In a later paper~\cite{tops-rings}, we will see that (5)--(6) are
indeed equivalent to (1)--(4) of Proposition~\ref{gh-reform}.
\end{remark}

\section{NIP rings of finite breadth}
If $R$ is a definable ring in some structure, a \emph{definable $R$-module} means a module $M$ such that the following things are definable in the structure:
\begin{itemize}
\item The underlying set of $M$.
\item The addition map $M \times M \to M$.
\item The scalar multiplication map $R \times M \to M$.
\end{itemize}
\begin{lemma} \label{extdef}
  Let $R$ be a definable ring and $M$ be a definable $R$-module in
  some structure.  Suppose $\br_R(M)$ is finite.  Then any
  $R$-submodule $M' \subseteq M$ is externally definable.
\end{lemma}
\begin{proof}
  Let $n = \br_R(M) < \infty$.  The submodule $M'$ is a directed union
  of its finitely generated submodules.  Each finitely-generated
  submodule is generated by at most $n$ elements.  In particular,
  finitely generated submodules of $M$ are uniformly definable.  A
  directed union of uniformly definable sets is externally definable
  \cite[Remark~2.9]{fpcase}.
\end{proof}
In this paper, an \emph{overring} of a domain $R$ is a subring of
$\Frac(R)$ containing $R$.  For example, the integral closure of $R$
is an overring of $R$.
\begin{corollary}\label{overrings}
  If $R$ is a $W_n$-domain and $S$ is an overring, then $S$ is
  externally definable as a subset of $\Frac(R)$, in the structure
  $R$.
\end{corollary}
\begin{proof}
  If $K = \Frac(R)$, then the breadth of $K$ as an $R$-module equals
  the breadth of $R$ by \cite[Lemma~2.4]{prdf5}.  In particular, $K$
  has finite breadth, so the $R$-submodule $S$ is externally
  definable.
\end{proof}
For example, the integral closure of $R$ is externally definable.  In
fact, the integral closure is genuinely definable:
\begin{lemma} \label{decompose}
  Let $R$ be a $W_n$-domain.  Let $\tilde{R}$ be the integral closure
  of $R$.  Then $\tilde{R} = \bigcap_{i = 1}^m \Oo_i$ for some
  valuation rings $\Oo_i$.  Moreover, $\tilde{R}$ and the $\Oo_i$ are
  definable subsets of $\Frac(R)$, in the structure $R$.
\end{lemma}
\begin{proof}
  Fact~\ref{tilde} gives the decomposition of $\tilde{R}$ as an
  intersection of valuation rings.  For definability of $\tilde{R}$,
  note that the following are equivalent for $x \in \Frac(R)$:
  \begin{enumerate}
  \item $x$ is integral over $R$.
  \item There is an overring $R \subseteq S \subseteq \Frac(R)$ such
    that $x \in S$ and $S$ is finitely-generated as an $R$-module.
  \item There is an overring $R \subseteq S \subseteq \Frac(R)$ such
    that $x \in S$ and $S$ has the form $Ry_1 + \cdots Ry_\ell$ for some
    $y_1,\ldots,y_\ell \in \Frac(R)$, where $\ell = \br(R)$.
  \end{enumerate}
  The equivalence of (1) and (2) is well-known; see
  \cite[Corollary~4.6]{eisenbud} for example.  The equivalence of (2)
  and (3) holds because the breadth of $\Frac(R)$ as an $R$-module is
  at most $n$ by \cite[Lemma~2.4]{prdf5}.  Condition (3) is a
  definable property.  This shows that $\tilde{R}$ is definable.  By
  Fact~\ref{multival}, the $\Oo_i$ are the localizations of
  $\tilde{R}$ at its maximal ideals.  The maximal ideals of
  $\tilde{R}$ are definable in $\tilde{R}$ by
  \cite[Corollary~2.4]{fpcase}, and so each $\Oo_i$ is definable.
\end{proof}
\begin{fact}[{\cite[Lemma~2.7]{prdf5}}] \label{whoops}
  If $R$ is a $W_n$-domain and $S$ is an overring
  %% (meaning $R
%%   \subseteq S \subseteq \Frac(R)$ in particular),
  then $S$ is a $W_n$-domain.
\end{fact}
\begin{remark} \label{closed2}
  Let $R$ be a ring and $I$ be an ideal.  Then $\br(R/I) \le \br(R)$,
  as the lattice of ideals in $R/I$ is a sublattice of the lattice of
  ideals in $R$.
\end{remark}
\begin{remark} \label{closed1}
  Let $R$ be a ring and $S$ be an $R$-algebra such that $S \cong R^d$
  as $R$-modules, for some finite $d$.  If $R$ has finite breadth, so
  does $S$, because
  \begin{equation*}
    \br(S) = \br_S(S) \le \br_R(S) = \br_R(R^d) = d \cdot \br_R(R) =
    d \br(R) < \infty
  \end{equation*}
  using Facts~\ref{base} and \ref{additive}.
\end{remark}
\begin{lemma}
  Let $\mathcal{K}$ be the class of rings of finite breadth.  Then
  $\mathcal{K}$ is a pre-henselizing class in the sense of Definition~\ref{good}.
\end{lemma}
\begin{proof}
  (Loc) holds by Fact~\ref{whoops}.  (El) is easy,
  as $\br(R) \le n$ is expressed by a first-order sentence.
  (Quot) holds by Remark~\ref{closed2}, and (Free)
  holds by Remark~\ref{closed1}.
\end{proof}
For the rest of the section, assume the henselianity conjecture (Conjecture~\ref{hens}).
\begin{fact}[Assuming HC] \label{henscor}
  Let $\Oo_1, \Oo_2$ be two valuation rings on a field $K$.  If the
  structure $(K,\Oo_1,\Oo_2)$ is NIP, then $\Oo_1$ and $\Oo_2$ are
  comparable.
\end{fact}
In fact, Fact~\ref{henscor} is \emph{equivalent} to the henselianity
conjecture, by \cite[Corollary~5.6]{hhj-v-top}.
\begin{lemma}[Assuming HC] \label{comparability}
  Let $R$ be an NIP $W_n$-domain.  Then $\Spec R$ is a chain.
\end{lemma}
\begin{proof}
  Suppose $\pp_1, \pp_2 \in \Spec R$ are incomparable.  Let $K =
  \Frac(R)$.  By Chevalley's extension theorem
  \cite[Theorem~3.1.1]{PE}, there exist valuation rings
  $(\Oo_i,\mm_i)$ on $K$ for $i = 1, 2$ such that $\Oo_i \supseteq R$
  and $\pp_i = \mm_i \cap R$.  Then $\Oo_1$ and $\Oo_2$ are
  incomparable.  (If, say, $\Oo_1 \subseteq \Oo_2$, then $\mm_1
  \supseteq \mm_2$ and $\pp_1 \supseteq \pp_2$, a contradiction.)  By
  Corollary~\ref{overrings}, $\Oo_1$ and $\Oo_2$ are both definable in
  $R^\Sh$, contradicting Fact~\ref{henscor}.
\end{proof}
In particular, the class of NIP rings of finite breadth is a henselizing
class (Definition~\ref{vgood}), and so Proposition~\ref{blackbox}
applies, giving the following:
\begin{theorem}[Assuming HC] \label{nip-w}
  Let $R$ be an NIP ring with $\br(R) < \infty$.
  \begin{enumerate}
  \item $R$ is a finite product of henselian local rings.
  \item If $R$ is an integral domain, then $R$ is a henselian local
    domain.
  \end{enumerate}
\end{theorem}
%% \begin{proof}
%%   The proofs of Corollary~3.16, Lemma~3.19, Proposition~3.20, and
%%   Theorems~3.21--3.22 in \cite{fpcase} apply with minimal changes,
%%   replacing ``NIP $\Ff_p$-algebra'' with ``NIP ring of finite
%%   breadth'', and using Lemma~\ref{comparability} instead of
%%   \cite[Theorem~3.15]{fpcase}.  In the proof of Corollary~3.16 we need
%%   to be able to pass to a quotient ring; this is okay by
%%   Remark~\ref{closed2}.  In the proof of Proposition~3.20, we need to
%%   be able to pass to an $R$-algebra $S$ such that $S$ is a finite-rank
%%   free $R$-module; this is okay by Remark~\ref{closed1}.
%% \end{proof}
Assuming the henselianity conjecture, we can also strengthen
Lemma~\ref{decompose}:
\begin{lemma}[Assuming HC] \label{decompose2}
  Let $R$ be an NIP domain with $\br(R) < \infty$.  Let $\tilde{R}$ be
  the integral closure of $R$.  Then $\tilde{R}$ is a henselian
  valuation ring, definable in $R$.
\end{lemma}
\begin{proof}
  By Lemma~\ref{decompose}, $\tilde{R}$ is definable and an
  intersection $\bigcap_{i = 1}^n \Oo_i$ of definable valuation rings,
  for some $n$.  We may assume the $\Oo_i$ are pairwise incomparable.
  Then $n = 1$ by Fact~\ref{henscor}, and $\Oo_1$ is henselian by the
  henselianity conjecture.
\end{proof}

\section{Dp-finite rings} \label{sec-moved-8.1-start}
Drop the assumption of the henselianity conjecture.

A \emph{definable topology} on a structure $M$ means
a topology such that some definable family $\{D_a\}_{a \in X}$ is a
basis of opens.  A topology on a field $K$ is a \emph{V-topology} if it is induced by a valuation or absolute value; see \cite[Section~3]{PZ} and \cite[Appendix~B]{PE} for more about V-topologies.  Definable V-topologies play an important role in the conjectural classification of NIP fields \cite{hhj-v-top}.

Recall that a structure is \emph{dp-finite} if its dp-rank is $n$ for
some natural number $n$.
\begin{fact}
  \phantomsection \label{df-hens}
  \begin{enumerate}
  \item If $\Oo$ is a dp-finite valuation ring, then $\Oo$ is
    henselian.
  \item If $(K,\Oo_1,\Oo_2)$ is a dp-finite bi-valued field, then
    $\Oo_1$ and $\Oo_2$ are comparable.
  \end{enumerate}
\end{fact}
\begin{proof}
  This is essentially \cite[Corollary~4.16]{prdf6}, though part (2)
  requires some explanation.  Part (1) is
  \cite[Corollary~4.16(3)]{prdf6}.  For part (2), suppose $\Oo_1$ and
  $\Oo_2$ are incomparable.  Then $\Oo_1 \cdot \Oo_2 = \{xy : x \in
  \Oo_1, ~ y \in \Oo_2\}$ is a definable valuation ring in the
  structure $(K,\Oo_1,\Oo_2)$.  Let $k$ be the residue field of $\Oo_1
  \cdot \Oo_2$.  The valuation rings $\Oo_1$ and
  $\Oo_2$ induce independent non-trivial valuations on $k$ (see \cite{PE}, specifically Theorem~2.3.4 and the discussion following Corollary~2.3.2).  Replacing
  $K$ with $k$, we may assume that $\Oo_1$ and $\Oo_2$ are independent
  non-trivial valuations.  Then the structure $(K,\Oo_1,\Oo_2)$ is unstable and has
  two definable V-topologies, contradicting
  \cite[Corollary~4.16(2)]{prdf6}.
\end{proof}

%The proof of Lemma~\ref{comparability} then gives
\begin{lemma} \label{comparability2}
  Let $R$ be a dp-finite $W_n$-domain.  Then $\Spec R$ is a chain.
\end{lemma}
\begin{proof}
  Suppose $\pp_1, \pp_2 \in \Spec R$ are incomparable.  As in the proof of Lemma~\ref{comparability},\footnote{This part of the proof of Lemma~\ref{comparability} did not use the assumption of the henselianity conjecture.} we can find two incomparable valuation rings $\Oo_1, \Oo_2$ on $K = \Frac(R)$ with $\Oo_1,\Oo_2 \supseteq R$.  By Corollary~\ref{overrings}, $\Oo_1$ and $\Oo_2$ are both definable in $R^{\Sh}$, contradicting Fact~\ref{df-hens}(2).
\end{proof}
Therefore the class of dp-finite rings of finite breadth is a henselizing class, and Proposition~\ref{blackbox} yields the following:
\begin{lemma} \label{lem:dpf}
  Let $R$ be a dp-finite ring with $\br(R) < \infty$.
  \begin{enumerate}
  \item $R$ is a direct product of finitely many henselian local
    rings.
  \item If $R$ is an integral domain, then $R$ is a henselian local
    domain and the prime ideals of $R$ are linearly ordered.
  \end{enumerate}
\end{lemma}
We will see shortly that we can remove the finite-breadth assumption
in Lemma~\ref{lem:dpf}.
\begin{lemma} \label{deflat}
  Let $N$ be some structure.  Let $R$ be a definable ring and $M$ be a
  definable $R$-module.  Let $\Sub^{\df}_R(M)$ be the modular lattice
  of \emph{definable} $R$-submodules of $M$.  Then $\br(M)$ is the
  breadth of $\Sub^{\df}_R(M)$.
\end{lemma}
\begin{proof}
  First of all,
  \begin{equation*}
    \br(\Sub^\df_R(M)) \le \br(\Sub_R(M)) \stackrel{\df}{=} \br(M),
  \end{equation*}
  as $\Sub^\df_R(M)$ is a sublattice of $\Sub_R(M)$.  It remains to show that
  \begin{equation*}
    n \le \br(M) \implies n \le \br(\Sub^\df_R(M))
  \end{equation*}
  for each finite $n$.  Take $x_1,\ldots,x_n \in M$ such that
  \begin{equation*}
    Rx_1 + \cdots + Rx_n \ne Rx_1 + \cdots + \widehat{Rx_i} + \cdots + Rx_n
  \end{equation*}
  for all $1 \le i \le n$.  Then the definable $R$-submodules
  $Rx_1,\ldots,Rx_n$ show that $\Sub^\df_R(M)$ has breadth at least
  $n$.
\end{proof}
\begin{remark} \label{infinite}
  If $M$ is a non-trivial $R$-module, then $M$ has a simple
  subquotient.  That is, there are submodules $M'' < M' \le M$ such
  that $M'/M''$ is a simple $R$-module, isomorphic to $R/\mm$ for some
  maximal ideal $\mm$.  To see this, take non-zero $x \in M$, take $M'
  = Rx \subseteq M$, and take $M''$ a maximal submodule of $M'$ not
  containing $x$.  As a consequence, if $R/\mm$ is infinite for every
  maximal ideal $\mm$, then every non-trivial $R$-module is infinite.
\end{remark}
\begin{lemma} \label{dprbr}
  Let $R$ be a dp-finite ring.  Suppose $R/\mm$ is infinite for every
  maximal ideal $\mm$ of $R$.  Then $\br(R) \le \dpr(R)$.  In
  particular, $\br(R)$ is finite.
\end{lemma}
\begin{proof}
  Suppose $\br(R) \ge n$.  By Lemma~\ref{deflat}, the lattice
  $\Sub^\df_R(R)$ of definable ideals has breadth at least $n$, so it
  contains a sublattice isomorphic to the powerset of $n$.  It follows
  that there are definable ideals $I < I' \le R$ and $J_1,\ldots,J_n$
  such that $I < J_i \le I'$, and $I'/I$ is an internal direct sum
  $\bigoplus_{i = 1}^n (J_i/I)$.  Since dp-rank is additive on products ($\dpr(X \times Y) = \dpr(X) + \dpr(Y)$), the dp-rank of the interpretable group $I'/I$ is given by $\dpr(I'/I) = \sum_{i=1}^n \dpr(J_i/I)$.  By Remark~\ref{infinite}, each
  interpretable group $J_i/I$ is infinite, and so $\dpr(J_i/I) \ge 1$.
  Then $\dpr(R) \ge \dpr(I') \ge \dpr(I'/I) = \sum_{i = 1}^n
  \dpr(J_i/I) \ge n$.  We have shown
  \begin{equation*}
    n \le \br(R) \implies n \le \dpr(R). \qedhere
  \end{equation*}
\end{proof}
\begin{example} \label{deH}
  If $(R,\mm)$ is a dp-minimal local domain and the residue field
  $R/\mm$ is infinite, then $\br(R) \le \dpr(R) = 1$, so $R$ is a
  $W_1$-domain, or equivalently, a valuation ring.  This was
  originally proven by d'Elb\'ee and Halevi
  \cite[Theorem~4.1]{halevi-delbee}.
\end{example}
\begin{lemma} \label{okay}
  Let $R$ be a problematic ring (Definition~\ref{vgood}), i.e., a domain with
  exactly two maximal ideals $\mm_1,\mm_2$, such that the quotients
  $R/\mm_1$ and $R/\mm_2$ are infinite.  Then $R$ is not dp-finite.
\end{lemma}
\begin{proof}
  Suppose otherwise.  By Lemma~\ref{dprbr}, $\br(R)$ is finite.  The
  incomparability of $\mm_1$ and $\mm_2$ contradicts
  Lemma~\ref{comparability2}.
\end{proof}
It follows that the class of dp-finite rings is henselizing.  Then
Proposition~\ref{blackbox} gives the following:
\begin{theorem} \label{dft}
  Let $R$ be a dp-finite ring.
  \begin{enumerate}
  \item $R$ is a direct product of finitely many henselian local rings.
  \item If $R$ is an integral domain, then $R$ is a henselian local
    domain and the prime ideals of $R$ are linearly ordered.
  \end{enumerate}
\end{theorem}
In the dp-minimal case, the fact that the prime ideals are linearly
ordered is due to d'Elb\'ee and Halevi
\cite[Corollary~2.4]{halevi-delbee}.

\section{Dp-finite $W_n$-domains}
Before focusing on Noetherian rings, there are a few more results we
can say about dp-finite rings of finite breadth.
\begin{proposition} \label{um}
  Let $R$ be a dp-finite $W_n$-domain with fraction field $K$.
  Suppose $R \ne K$.  Let $\tilde{R}$ be the integral closure of $R$.
  \begin{enumerate}
  \item $\tilde{R}$ is a henselian valuation ring, definable in the
    structure $R$ as a subset of $\Frac(R)$.
  \item The structure $(K,+,\cdot,R,\tilde{R})$ has the same dp-rank
    as $R$.
  \end{enumerate}
  Recall that $R$ is local with maximal ideal $\mm$.  Let
  $\tilde{\mm}$ denote the maximal ideal of $\tilde{R}$.
  \begin{enumerate}[resume]
  \item The inclusion $R \to \tilde{R}$ is a local homomorphism,
    meaning that $\mm = R \cap \tilde{\mm}$.
  \item $R$ and $\tilde{R}$ have the same residue characteristic.
  \item $R/\mm$ is finite iff $\tilde{R}/\tilde{\mm}$ is finite.  More generally, $\tilde{R}/\tilde{\mm}$ is a finite extension of $R/\mm$.
  \end{enumerate}
\end{proposition}
\begin{proof}
  \begin{enumerate}
  \item Lemma~\ref{decompose} shows that $\tilde{R}$ is a finite intersection of definable valuation rings $\bigcap_{i=1}^m \Oo_i$.  By Fact~\ref{df-hens}(1), the $\Oo_i$ are henselian.  By Fact~\ref{df-hens}(2), the $\Oo_i$ are pairwise comparable, so we can take $m=1$.
  \item By \cite[Proposition~2.8(2)]{halevi-delbee}, $(K,+,\cdot,R)$ has the same
    dp-rank as $R$.  As $\tilde{R}$ is definable in this
    structure, the expansion $(K,+,\cdot,R,\tilde{R})$ has the same
    dp-rank. % by Remark~\ref{shelah-rem}.
  \item Suppose $x \in R$.  We must show $x \in \mm \Leftrightarrow x
    \in \tilde{\mm}$.  One direction is easy: if $x \notin \mm$ then
    $x \in R^\times \subseteq \tilde{R}^\times$, so $x \notin
    \tilde{\mm}$.  Conversely, suppose $x \notin \tilde{\mm}$ but $x
    \in \mm$.  As $\tilde{R}$ is a valuation ring, $1/x \in
    \tilde{R}$, meaning that $1/x$ is integral over $R$.  Therefore
    there are $c_0, \ldots, c_{n-1} \in R$ such that $x^{-n} = c_0 +
    c_1 x^{-1} + \cdots + c_{n-1} x^{1-n}$.  Then $1 = c_0x^n + \cdots
    + c_{n-1}x \in \mm$, which is absurd.
  \item Part (3) gives a field embedding $R/\mm \to
    \tilde{R}/\tilde{\mm}$, so these must have the same
    characteristic.
  \item The $(\Leftarrow)$ direction follows from the field embedding
    $R/\mm\to \tilde{R}/\tilde{\mm}$.  Conversely, suppose $R/\mm$ is
    finite.  Note that $\tilde{R}$ and $\tilde{\mm}$ are
    $\tilde{R}$-submodules of $K$, hence $R$-submodules of $K$ as $R
    \subseteq \tilde{R}$.  By \cite[Lemma~2.4]{prdf5}, $\br_R(K) =
    \br_R(R) < \infty$, and so
    \begin{equation*}
      \br_R(\tilde{R}/\tilde{\mm}) \le \br_R(\tilde{R}) \le \br_R(K) =
      \br(R) < \infty
    \end{equation*}
    by Fact~\ref{additive}.  We claim that $\mm$ annihilates the
    $R$-module $\tilde{R}/\tilde{\mm}$.  Indeed, if $x \in \mm$ and $y
    \in \tilde{R}$, then $x \in \tilde{\mm}$ by part (3), and so $xy
    \in \tilde{\mm} \tilde{R} = \tilde{\mm}$.  Therefore
    $\tilde{R}/\tilde{\mm}$ is an $R/\mm$-module, and
    \begin{equation*}
      \dim_{R/\mm}(\tilde{R}/\tilde{\mm}) =
      \br_{R/\mm}(\tilde{R}/\tilde{\mm}) =
      \br_R(\tilde{R}/\tilde{\mm}) < \infty
    \end{equation*}
    by Remarks~\ref{base2} and \ref{vec-br}.  Therefore if $R/\mm$ is finite then so is
    $\tilde{R}/\tilde{\mm}$. \qedhere
  \end{enumerate}
\end{proof}
\begin{corollary}
  Let $R$ be a dp-finite $W_n$-domain with $R \ne \Frac(R)$.  If
  $\Frac(R)$ has positive characteristic, then the residue field of
  $R$ is infinite and $\{x^p - x : x \in R\} = R$.
\end{corollary}
\begin{proof}
  Let $\tilde{R}$ be the integral closure.  Then $\tilde{R}$ is a
  positive characteristic NIP valuation ring, so it has infinite
  residue field by~\cite[Proposition~5.3]{KSW}.  By Proposition~\ref{um},
  $R$ has infinite residue field, too.  The surjectivity of the
  Artin-Schreier map $R \to R$ then follows by
  \cite[Theorem~3.4]{fpcase}.
\end{proof}

\section{NIP Noetherian rings}
\begin{lemma} \label{semisimple}
  Let $R$ be a ring.  Then $\br(R) \ge n$ iff there are ideals $I \le
  I' \le R$ such that the $R$-module $I'/I$ is a direct sum of $n$
  simple $R$-modules.
\end{lemma}
\begin{proof}
  The condition certainly implies $\br(R) \ge n$.  Conversely, supose
  $\br(R) \ge n$.  Then there are ideals $I_0 \le I'_0 \le R$ such
  that the $R$-module $I'_0/I_0$ is a direct sum of $n$ non-trivial
  $R$-modules.
  \begin{equation*}
    I'_0/I_0 \cong \bigoplus_{i=1}^n N_i.
  \end{equation*}
  Every non-trivial $R$-module has a simple subquotient (see
  Remark~\ref{infinite}), so there are submodules $N'_i \le N''_i \le
  N_i$ such that $N''_i/N'_i$ is simple.  Then
  \begin{equation*}
    \left. \bigoplus_{i=1}^n N''_i \middle/ \bigoplus_{i=1}^n N'_i \right. \cong \bigoplus_{i=1}^n N''_i/N'_i
  \end{equation*}
  is a subquotient of $I'_0/I_0$ isomorphic to a direct sum of $n$
  simple $R$-modules.  A subquotient of $I'_0/I_0$ has the form $I'/I$
  for some ideals $I, I'$ with $I_0 \le I \le I' \le I'_0$.
\end{proof}
If $R$ is a Noetherian ring, let $\dim(R)$ denote the Krull dimension
of $R$.  Recall
that if $R$ is a Noetherian ring of Krull dimension 0, then $R$ has
finite length (as an $R$-module) \cite[Theorem~2.14]{eisenbud}.
%% See
%% \cite[Section~2.4]{eisenbud} for more about length of modules.
A
ring $R$ is \emph{semilocal} if it has finitely many maximal ideals
\cite[Exercise~4.13]{eisenbud}.
\begin{lemma} \label{noether-mess}
  Let $R$ be a semilocal Noetherian domain with $\dim(R) \le 1$.  Then
  $\br(R)$ is finite.
\end{lemma}
\begin{proof}
  If $\dim(R) = 0$, then $R$ is a field, and $\br(R) = 1$.  Suppose
  $\dim(R) = 1$.  Let $\mm_1,\ldots,\mm_n$ be the maximal ideals of
  $R$.  The only other prime ideal is the minimal prime ideal $0$.

  Note that if $I$ is any ideal in $R$, then the poset $\Spec R/I$ is
  isomorphic to the subposet $\{\pp \in \Spec R : \pp \supseteq I\}$
  of $\Spec R$.  In the case when $I \ne 0$, this poset is a subset of
  the antichain $\{\mm_1,\ldots,\mm_n\}$, so $\dim(R/I) = 0$ and $R/I$
  has finite length.  Letting $\ell_R(M)$ denote the length of an $R$-module $M$, we have
  \begin{equation}
    I \ne 0 \implies \ell_R(R/I) = \ell_{R/I}(R/I) < \infty. \tag{$\ast$}
  \end{equation}
  Let $J = \bigcap_{i = 1}^n \mm_i$ (the Jacobson radical); it is
  non-zero because an intersection of two non-zero ideals is non-zero
  in a domain.  Take non-zero $x \in J$.  By ($\ast$), the $R$-module
  $R/xR$ has finite length $m$.  We claim $\br(R) \le m$.  Otherwise
  by Remark~\ref{semisimple} there are ideals $I' < I \le R$ such that
  $I/I'$ is a direct sum of $m+1$ simple $R$-modules.  Recall that any simple $R$-module $M$ is isomorphic to $R/\mm$ for some maximal ideal $\mm$, and so the annihilator $\Ann(M) = \mm$ is a maximal ideal.  Because $x$ is
  in every maximal ideal, it annihilates any simple $R$-module, and
  therefore it annihilates $I/I'$, meaning that $xI \le I'$.  Then
  \begin{equation*}
    m+1 = \ell_R(I/I') \le \ell_R(I/xI).
  \end{equation*}
  By ($\ast$), the modules $R/I$ and $R/xI$ have finite length, and so
  \begin{equation*}
    \ell_R(R/I) + \ell_R(I/xI) = \ell_R(R/xI) = \ell_R(R/xR) + \ell_R(xR/xI) < \infty
  \end{equation*}
  by the additivity of length (Fact~\ref{JH}).
  Multiplication by $x$ gives an isomorphism of $R$-modules from $R/I$
  to $xR/xI$, so $\ell_R(R/I) = \ell_R(xR/xI)$.  Canceling this from both
  sides, we see $\ell_R(I/xI) = \ell_R(R/xR) = m$.  Then
  \begin{equation*}
    m + 1 = \ell_R(I/I') \le \ell_R(I/xI) = m,
  \end{equation*}
  a contradiction.
\end{proof}
\begin{lemma} \label{mess2}
  Let $R$ be a semilocal Noetherian ring with $\dim(R) \le 1$.  Then
  $\br(R)$ is finite.
\end{lemma}
\begin{proof}
  Recall that Noetherian rings have only finitely many minimal prime ideals \cite[Section~3.1]{eisenbud}.
  Let $\pp_1, \ldots, \pp_n$ denote the finitely many minimal prime
  ideals of $R$.  By Lemma~\ref{noether-mess}, the ring $R/\pp_i$ has
  finite breadth.  By Fact~\ref{additive}, it follows that any
  finitely-generated $R/\pp_i$-module has finite breadth as an
  $R/\pp_i$-module, or equivalently, as an $R$-module.

  By a basic fact in commutative algebra
  \cite[Corollary~2.12]{eisenbud}, the set of nilpotent elements
  (i.e., the nilradical $\sqrt{0}$) is exactly $\bigcap_{i = 1}^n
  \pp_i$.  In particular, the product $I = \pp_1 \pp_2 \cdots \pp_n$
  is contained in the nilradical---every element of $I$ is nilpotent.
  As $I$ is finitely-generated, it follows that $I^k = 0$ for some
  $k$.  Thus $\pp_1^k \pp_2^k \cdots \pp_n^k = 0$ for some $k$.
  Therefore there is a finite sequence $\qq_1,\ldots,\qq_m$ such that
  each $\qq_i$ is a minimal prime, and $\qq_1 \qq_2 \cdots \qq_m = 0$.
  Let $I_j = \qq_1 \cdots \qq_j$ for $0 \le j \le m$.  This gives a
  descending sequence of ideals
  \begin{equation*}
    R = I_0 \ge I_1 \ge \cdots \ge I_N = 0.
  \end{equation*}
  For each $j > 0$, $I_{j-1}/I_j = I_{j-1}/\qq_j I_{j-1}$ is an
  $R/\qq_j$-module, finitely generated by Noetherianity.  By the first
  paragraph of the proof, $I_{j-1}/I_j$ has finite breadth as an
  $R$-module.  By Fact~\ref{additive} again, $R = I_0/I_N$ has finite
  breadth, bounded by $\sum_{j = 1}^N \br(I_j/I_{j-1})$.
\end{proof}
\begin{remark}
  Lemma~\ref{mess2} is related to---and possibly a consequence of---work of I. S. Cohen.
  In \cite[Section~4]{cohen-rm}, Cohen says that a ring $R$ has ``rank
  $k$'' if every ideal $I \subseteq R$ is generated by $k$ or fewer
  generators.  This condition is at first glance similar to the
  condition $\br(R) \le k$, which means that any finite set of
  generators for an ideal can be shrunk to a set of size $k$.  In
  general the two conditions are orthogonal.  For example, $\Zz$ has
  infinite breadth, but ``rank 1'' in the sense of Cohen.  Conversely,
  if $R$ is a non-discrete valuation ring, then $R$ has breadth 1, but
  does not have ``finite rank'' in the sense of Cohen, as $R$ is
  non-Noetherian.  When $R$ is a Noetherian ring, $\br(R) \le k$
  implies Cohen's ``rank $k$''.  In fact, when $(R,\mm)$ is a local
  Noetherian ring, ``rank $k$'' is equivalent to ``$\br(R) \le k$''
  because Nakayama's lemma makes both conditions equivalent to the
  condition
  \begin{equation*}
    \dim_{R/\mm}(I/\mm I) \le k \text{ for any ideal } I \subseteq R.
  \end{equation*}
  Cohen shows that if $R$ is a local Noetherian domain, then $R$ has
  finite rank if and only if $\dim(R) \le 1$
  \cite[Theorem~9]{cohen-rm} (see also \cite[Corollary~1 to
    Theorem~1]{cohen-rm}).  This gives another proof of
  Lemma~\ref{mess2}, at least in the case when $R$ is a local domain.
  (The direction of Cohen's result that is relevant to
  Lemma~\ref{mess2} may go back to \cite{akizuki}.)
\end{remark}

%% \begin{remark} \label{width-rem}
%%   If $R$ is a ring and $\br(R) \le n$, then the width of $\Spec R$ is
%%   at most $n$.  Indeed, suppose $\pp_1,\ldots,\pp_{n+1} \in \Spec R$
%%   are incomparable.  Then
%%   \begin{equation*}
%%     \pp_1 \cap \cdots \cap \pp_{n+1} \ne \pp_1 \cap \cdots \cap
%%     \widehat{\pp_i} \cap \cdots \cap \pp_{n+1}
%%   \end{equation*}
%%   for any $i$, by commutative algebra.  Therefore the lattice of
%%   ideals has breadth at least $n+1$, meaning that $\br(R) \ge n+1$.
%% \end{remark}
\begin{proposition} \label{phew}
  Let $R$ be a Noetherian ring.  The following are equivalent:
  \begin{enumerate}
  \item $\br(R) < \infty$.
  \item The poset $\Spec R$ has finite width.
  \item $\Spec R$ is finite.
  \item $R$ is semilocal with $\dim(R) \le 1$.
  \end{enumerate}
\end{proposition}
\begin{proof}
  \begin{description}
  \item[$(1)\implies(2).$] Remark~\ref{width-rem}.
  \item[$(2)\implies(3).$] By Dilworth's theorem~\cite{dilworth}, $\Spec R$ is a
    finite union of chains.  In a Noetherian ring, the poset $\Spec R$
    satisfies the ascending chain condition (by Noetherianity) and the
    descending chain condition (by dimension theory; see~\cite[Corollary~10.3]{eisenbud}).  Therefore each chain in the union is finite.
  \item[$(3)\implies(4).$] The fact that $R$ is semilocal is clear.
    Suppose for the sake of contradiction that $\dim(R) \ge 2$.  Take
    three prime ideals $\pp_1 \subsetneq \pp_2 \subsetneq \pp_3$.  By
    \cite[Theorem~144]{kaprings}, the fact that there is a prime ideal
    between $\pp_1$ and $\pp_3$ implies that there are infinitely many
    such prime ideals.  This contradicts (3).
  \item[$(4)\implies(1).$] Lemma~\ref{mess2}.    \qedhere
  \end{description}
\end{proof}
\begin{corollary} \label{hah}
  Let $R$ be an NIP Noetherian ring.  Then $R$ has finite breadth, $R$
  is semilocal with $\dim(R) \le 1$, and $\Spec R$ is finite.
\end{corollary}
\begin{proof}
  If $R$ is an NIP ring, then $\Spec R$ has finite width by Fact~\ref{fin-wid}.
\end{proof}
\begin{lemma} \label{where}
  Let $R$ be an NIP Noetherian domain with fraction field $K$.  Let
  $\tilde{R}$ be the integral closure of $R$ (in $K$).
  \begin{enumerate}
  \item $\tilde{R}$ is a finite intersection of discrete valuation
    rings, each of which is definable in $R$.
  \item If $R$ is dp-finite or the henselianity conjecture holds, then
    $\tilde{R}$ is a henselian DVR.
  \end{enumerate}
\end{lemma}
\begin{proof}
  By Corollary~\ref{hah}, $R$ has finite breadth and $\dim(R) \le 1$.
  \begin{enumerate}
  \item By Lemma~\ref{decompose}, $\tilde{R}$ is an intersection of
    definable valuation rings $\Oo_1 \cap \cdots \cap \Oo_n$.  Because $\dim(R) \le 1$, the Krull-Akizuki
    theorem \cite[Theorem~11.13]{eisenbud} shows that all overrings of $R$ are
    Noetherian and 1-dimensional.  Therefore each valuation ring
    $\Oo_i$ is Noetherian, and therefore a discrete valuation ring.
  \item By Proposition~\ref{um}(1) or Lemma~\ref{decompose2},
    $\tilde{R}$ is a henselian valuation ring. \qedhere
  \end{enumerate}
\end{proof}

%%   Then
%%   $\tilde{R}$ is a finite intersection of externally definable discrete
%%   valuation rings.
%% \end{lemma}
%% \begin{proof}
%%   By Corollary~\ref{hah}, $R$ has finite breadth.  Then $\tilde{R} =
%%   \bigcap_{i = 1}^m \Oo_i$ for some valuation rings $\Oo_i$ by
%%   Lemma~\ref{decompose}, and $\tilde{R}$ and the $\Oo_i$ are all
%%   externally definable.  Omitting irrelevant terms in the
%%   intersection, we may assume the $\Oo_i$ are pairwise incomparable.
%%   By Fact~\ref{multival},
%%   $\tilde{R}$ has finitely many maximal ideals $\mm_1,\ldots,\mm_m$,
%%   and the $\Oo_i$ are the localizations $R_{\mm_i}$.

%%     By the Krull-Akizuki
%%   theorem \cite[Theorem~11.13]{eisenbud}, the integral closure
%%   $\tilde{R}$ and the localizations $\tilde{R}_{\mm_i}$ are
%%   Noetherian.  Therefore the localizations $R_{\mm_i}$ are DVRs.
%% \end{proof}

\begin{theorem} \label{charzero}
  Let $R$ be an NIP Noetherian domain.  If $R$ is not a field, then
  $\Frac(R)$ has characteristic 0.
\end{theorem}
\begin{proof}
  By Lemma~\ref{where}, the integral closure $\tilde{R}$ is an
  intersection of discrete valuation rings $\bigcap_{i = 1}^m \Oo_i$,
  with each $\Oo_i$ definable in the structure $R$.  If $R \ne K$,
  then $\tilde{R} \ne K$.  (Take a non-zero maximal ideal $\mm \in
  \Spec R$.  Take $x \in \mm$.  Then $1/x$ is not integral over $R$,
  or else $x^{-n} + c_{n-1} x^{1-n} + \cdots + c_1 x^{-1} + c_0 = 0$
  for some $c_i$ in $R$.  This would imply $-1 = c_{n-1}x + c_{n-2}x^2
  + \cdots + c_0x^n \in \mm$, which is absurd.)  Therefore some
  $\Oo_i$ is not $K$.\footnote{More precisely, since $K$ is not a DVR,
  the claim is that $m > 0$.}  Then there is at least one DVR $\Oo$ on
  $K$, definable in the structure $R$.  By
  \cite[Proposition~5.4]{KSW}, there are no NIP discrete valuation
  rings of positive characteristic, so $\characteristic(K) =
  0$.
\end{proof}
Now assume the henselianity conjecture  (Conjecture~\ref{hens}).
\begin{theorem}[Assuming HC] \label{xyz}
  Let $R$ be an NIP Noetherian ring.
  \begin{enumerate}
  \item $R$ is a finite product of henselian local rings.
  \item If $R$ is an integral domain, then $R$ is a henselian local
    domain.
  \end{enumerate}
\end{theorem}
\begin{proof}
  This follows directly from Theorem~\ref{nip-w}, since we know
  $\br(R) < \infty$ by Corollary~\ref{hah}.
\end{proof}

\subsection{Rings elementarily equivalent to NIP Noetherian rings}
Drop the assumption of the henselianity conjecture.
\begin{proposition} \label{udef}
  Let $R$ be an NIP Noetherian ring.  The family of ideals in $R$ is
  uniformly definable.
\end{proposition}
\begin{proof}
  Let $n = \br(R) < \infty$.  By Noetherianity, every ideal is
  finitely generated.  Then every ideal is generated by $n$ or fewer
  generators.
\end{proof}
\begin{proposition}
  Let $R_0$ be an NIP Noetherian ring, and suppose $R \equiv R_0$.
  \begin{enumerate}
  \item $R$ has finite breadth $n < \infty$.
  \item An ideal $I \subseteq R$ is definable if and only if it is
    finitely generated if and only if it is generated by $n$ or fewer
    generators.
  \item Any ideal in $R$ is externally definable.
  \item The family $\mathcal{I}$ of definable ideals in $R$ is uniformly definable.
  \item If $\mathcal{F}$ is a non-empty definable family of ideals,
    then $\mathcal{F}$ has a maximal element.
  \end{enumerate}
\end{proposition}
\begin{proof}
  The condition ``$\br(R) \le n$'' is expressed by a first-order
  sentence, so breadth is preserved under elementary equivalence.
  Therefore (1) holds by Corollary~\ref{hah}.  Then (3) holds by
  Lemma~\ref{extdef}.

  For (2), the fact that $\br(R) = n$ shows that every finitely
  generated ideal is generated by $n$ elements.  Finitely generated
  ideals are certainly definable.  Conversely, suppose $I \subseteq R$
  is definable, defined by a formula $\phi(x,b)$ with parameters $b
  \in I$.  As in the proof of Proposition~\ref{udef}, the ring $R_0$
  satisfies the following property:
  \begin{quote}
    For any $b'$, if $\phi(R_0,b')$ is an ideal, then $\phi(R_0,b')$
    is generated by $n$ elements.
  \end{quote}
  This is expressed by a first-order sentence, so it holds in $R$, and
  therefore $\phi(R,b)$ is generated by $n$ elements.

  Once (2) is known, (4) follows immediately.  For part (5), note that
  the ring $R_0$ satisfies the following property, by Noetherianity:
  \begin{quote}
    Let $\mathcal{I}_n$ be the family of ideals generated by $n$
    elements.  If $\mathcal{F}$ is a non-empty definable subfamily of
    $\mathcal{I}_n$, then $\mathcal{F}$ has a maximal element.
  \end{quote}
  This can be expressed by an axiom schema, so it holds in $R$.  But
  in $R$, every definable ideal is generated by $n$ elements, by part
  (2).  Therefore (5) holds.
\end{proof}
\begin{proposition}
  Let $R$ be a Noetherian ring and $R_0$ be an elementary
  substructure.  Then $R_0$ is Noetherian.
\end{proposition}
\begin{proof}
  Otherwise, take an ideal $I \subseteq R_0$ that is not finitely
  generated.  Recursively choose $x_1, x_2, \ldots \in I$ by taking
  $x_n \in I \setminus (x_1,\ldots,x_{n-1}) \ne \varnothing$.  Then we
  have an ascending chain of finitely-generated ideals in $R_0$:
  \begin{equation*}
    0 \subsetneq (x_1) \subsetneq (x_1,x_2) \subsetneq \cdots
  \end{equation*}
  In the ring $R$, the chain
  \begin{equation*}
    0 \subseteq (x_1)_R \subseteq (x_1,x_2)_R \subseteq \cdots
  \end{equation*}
  cannot be strictly ascending, because $R$ is Noetherian.  Therefore
  there is some $n$ such that $x_n \in (x_1,\ldots,x_{n-1})_R = Rx_1 +
  \cdots + Rx_{n-1}$.  As $R_0 \preceq R$, we also have $x_n \in
  R_0x_1 + \cdots + R_0x_{n-1}$, contradicting the choice of $x_n$.
\end{proof}

\section{Dp-finite Noetherian domains}
\begin{fact}[{part of \cite[Theorem~2.11]{dpm2}}] \label{trickle}
  Let $(\Oo,\mm)$ be a dp-finite (or strongly dependent) non-trivially
  valued field with residue field $k$ and value group $\Gamma$.
  Suppose $\characteristic(k) = p > 0$.  Then one of three things
  happens:
  \begin{enumerate}
  \item $\characteristic(K) = 0$, $k$ is finite, and the interval
    $[-v(p),v(p)] \subseteq \Gamma$ is finite.
  \item $\characteristic(K) = 0$, $k$ is infinite, and the interval
    $[-v(p),v(p)] \subseteq \Gamma$ is $p$-divisible.
  \item $\characteristic(K) = p$, $k$ is infinite, and $\Gamma$ is
    $p$-divisible.
  \end{enumerate}
\end{fact}
Let $R$ be a dp-finite Noetherian domain with $R \ne \Frac(R) = K$.
Let $\tilde{R}$ be the integral closure of $R$.  Here is what we know
so far:
\begin{itemize}
\item $\characteristic(K) = 0$ (Theorem~\ref{charzero}).
\item $R$ is a henselian local domain (Theorem~\ref{dft}).
\item $R$ has Krull dimension 1 and finite breadth (Corollary~\ref{hah}).
\item $\tilde{R}$ is a definable henselian DVR
  (Lemma~\ref{where}).
\item The structure $(K,+,\cdot,R,\tilde{R})$ has the same dp-rank as
  $R$ (Proposition~\ref{um}(2)).  
\item $R$ and $\tilde{R}$ have the same residue characteristic (Proposition~\ref{um}(4)).
\item The residue field of $R$ is finite iff the residue field of
  $\tilde{R}$ is finite (Proposition~\ref{um}(5)).
\end{itemize}
As $\tilde{R}$ is definable, it is also dp-finite.  By
Fact~\ref{trickle} and the fact that $\tilde{R}$ is a discrete
valuation ring, one of two things happens:
\begin{enumerate}
\item $\tilde{R}$ has residue characteristic 0.
\item $\tilde{R}$ has finite residue field.
\end{enumerate}
(Cases (2)--(3) of Fact~\ref{trickle} cannot happen because the valuation is
discrete.)  By parts (4) and (5) of Proposition~\ref{um}, the
properties of the residue field of $\tilde{R}$ transfer back to the
residue field of $R$.  Putting everything together, we get the
following trichotomy for dp-finite Noetherian domains:
\begin{theorem} \label{tricho1}
  Let $R$ be a dp-finite Noetherian domain with fraction field $K$.
  Recall that $R$ is local; let $k$ be the residue field.  Then one of
  three things happpens:
  \begin{enumerate}
  \item $R$ is a field.
  \item $R$ is not a field.  $K$ and $k$ have characteristic 0.
  \item $R$ is not a field.  $K$ has characteristic 0 and $k$ is
    finite.
  \end{enumerate}
\end{theorem}
\begin{remark}
  Theorem~\ref{tricho1} does not hold for general NIP Noetherian
  rings.  Let $\Qp^{un}$ be the maximal unramified algebraic extension
  of $\Qp$.  Then $\Qp^{un}$ is NIP \cite[Corollaire~7.5]{belair}.  If
  $R$ is the ring of integers of $\Qp^{un}$, then $R$ does not satisfy
  Theorem~\ref{tricho1}, because the residue field $\Ff_p^{alg}$ is
  infinite with positive characteristic.  The field $\Qp^{un}$ and
  ring $R$ are not dp-finite, or even strongly dependent by
  Fact~\ref{trickle}.
\end{remark}
\subsection{The classification of dp-minimal Noetherian domains} \label{ssec8.1}
%% In the following, a \emph{definable topology} on a structure $M$ means
%% a topology such that some definable family $\{D_a\}_{a \in X}$ is a
%% basis of opens.
If $R$ is a ring and $A, B \subseteq R$, let
\begin{gather*}
  A + B = \{x + y : x \in A, ~ y \in B\} \\
  A - B = \{x - y : x \in A, ~ y \in B\}.
\end{gather*}
In particular, $A - B$ does not mean $A \setminus B$.  Recall the definition of ``definable V-topology'' from the start of Section~\ref{sec-moved-8.1-start}.
\begin{fact}[{\cite[Theorem~1.3]{dpm1}}] \label{cantop}
  Let $(K,+,\cdot,\ldots)$ be an expansion of a field.  Suppose $K$ is
  dp-minimal but not strongly minimal.  Then there is a non-discrete,
  Hausdorff, definable field topology $\tau_K$ on $K$ characterized by
  the fact that the following family is a neighborhood basis of $0$:
  \begin{equation*}
    \{X - X : X \subseteq K \text{ is definable and infinite}\}.
  \end{equation*}
  Moreover, $\tau_K$ is a V-topology.
\end{fact}
The topology $\tau_K$ is called the \emph{canonical topology} on
$(K,+,\cdot,\ldots)$.
\begin{lemma} \label{only-top}
  Let $(K,+,\cdot,\ldots)$ be as in Fact~\ref{cantop}.  Let $\tau_0$
  be a non-discrete, Hausdorff, definable ring topology on $K$.  Then
  $\tau_0$ is the canonical topology $\tau_K$.
\end{lemma}
\begin{proof}
  Among non-discrete Hausdorff ring topologies on $K$, V-topologies
  are maximal with respect to coarsening---there are no strictly coarser ring topologies \cite[Theorem~3.2]{PZ}.  As $\tau_K$ is a
  V-topology, it suffices to show that $\tau_0$ is coarser than
  $\tau_K$.  Let $U$ be a $\tau_0$-neighborhood of 0.  It suffices to
  show that $U$ is a $\tau_K$-neighborhood of 0.  By
  $\tau_0$-continuity of $+, -$, there is a $\tau_0$-neighborhood $V$
  of $0$ such that $V - V \subseteq U$.  As $\tau_0$ is definable, we
  can take $V$ to be a definable set.  As $\tau_0$ is a non-discrete
  ring topology, $0 \in K$ is not $\tau_0$-isolated.  Therefore $V$ is
  infinite.  Then $V - V$ is a $\tau_K$-neighborhood of 0 by
  definition of $\tau_K$, and so $U$ is a $\tau_K$-neighborhood of 0.
\end{proof}
If $R \ne K = \Frac(R)$, then $R$ induces a ring topology on $K$
characterized by the fact that $\{aR : a \in K^\times\}$ is a
neighborhood basis of 0, or equivalently, characterized by the fact
that the set of non-zero ideals of $R$ form a neighborhood basis of 0
\cite[Example~1.2]{PZ}.  Following \cite{mana}, we call this topology
the \emph{$R$-adic topology}.
\begin{lemma} \label{moreover}
  Suppose $R \ne K = \Frac(R)$ and $R$ is dp-minimal and Noetherian.
  Let $\tilde{R}$ be the integral closure of $R$.  Then the $R$-adic
  topology agrees with the $\tilde{R}$-adic topology.  Moreover there
  is $a \in K^\times$ such that $R \supseteq a \tilde{R}$.
\end{lemma}
\begin{proof}
  The ring $R$ has finite breadth by Corollary~\ref{hah}, and so
  the structure $(K,R,\tilde{R})$ is dp-minimal by
  Proposition~\ref{um}(2).  By Lemma~\ref{only-top}, the $R$-adic
  topology and $\tilde{R}$-adic topology must both equal the canonical
  topology on $(K,R,\tilde{R})$.  Then $R$ is a neighborhood of 0 with
  respect to the $\tilde{R}$-adic topology, so there is $a \in
  K^\times$ such that $R \supseteq a \tilde{R}$.
\end{proof}
Recall the classification of dp-minimal DVRs mentioned in the
introduction (Fact~\ref{dvr-class}), including the division into
equicharacteristic and mixed characteristic cases.  Note that the
equicharacteristic cases always have characteristic 0.

\begin{lemma} \label{finquot}
  Let $R$ be a mixed characteristic dp-minimal DVR.  If $a \in R
  \setminus \{0\}$, then $R/aR$ is finite.
\end{lemma}
\begin{proof}
  Let $\pi \in R$ generate the maximal ideal.  Then $a = \pi^n u$
  where $n$ is the valuation of $a$ and $u$ is a unit, and so $R/aR =
  R/\pi^n R$.  We claim that $R/\pi^n R$ is finite by induction on
  $n$.  The base case $n = 1$ holds by Fact~\ref{dvr-class}.  If $n >
  0$, then we have a short exact sequence
  \begin{equation*}
    0 \to \pi R / \pi^n R \to R / \pi^n R \to R / \pi R \to 0
  \end{equation*}
  and an isomorphism $\pi R / \pi^n R \to R / \pi^{n-1} R$.  By
  induction, $R / \pi R$ and $R / \pi^{n-1} R$ are finite, so $R /
  \pi^n R$ is finite.
\end{proof}
If $R$ is a ring, a \emph{finite-index subring} is a subring $S
\subseteq R$ with finite index in the additive group $(R,+)$.
\begin{lemma} \label{yes}
  Suppose $R$ is a mixed characteristic dp-minimal DVR and $S$ is a
  finite index subring.
  \begin{enumerate}
  \item $S$ is dp-minimal.
  \item $S$ is Noetherian.
  \end{enumerate}
\end{lemma}
\begin{proof}
  \begin{enumerate}
  \item It suffices to show that $S$ is definable in $R$.  Let $n$ be
    the index of $S$ in $R$.  Then $n$ annihilates $R/S$, so $S
    \supseteq nR$.  As $\Frac(R)$ has characteristic 0, $R \models n
    \ne 0$ and so $R/nR$ is finite by Lemma~\ref{finquot}.  Then $S$ is
    a finite union of additive cosets of $nR$, so $S$ is definable.
  \item As in the previous point, $S \supseteq nR$ for some $n$.  Let
    $I$ be an ideal in $S$.  We claim that $I$ is finitely generated.
    We may assume $I \ne 0$.  Take non-zero $a \in I \subseteq S
    \subseteq R$.  Then
    \begin{equation*}
      anR \subseteq aS \subseteq I \subseteq S \subseteq R.
    \end{equation*}
    By Lemma~\ref{finquot}, $R/anR$ is finite.  Therefore $I/aS$ is
    finite.  Let $\{b_1,\ldots,b_m\}$ be a finite set of coset
    representatives of $aS$ in $I$.  Then $I$ is generated as an
    $S$-ideal by $\{a,b_1,\ldots,b_m\}$. \qedhere
  \end{enumerate}
\end{proof}
We can now classify dp-minimal Noetherian domains.
\begin{theorem} \label{class}
  The dp-minimal Noetherian domains are precisely the following:
  \begin{enumerate}
  \item Dp-minimal fields.
  \item Equicharacteristic dp-minimal DVRs.
  \item Finite-index subrings of mixed characteristic dp-minimal DVRs.
  \end{enumerate}
\end{theorem}
\begin{proof}
  The three cases are all dp-minimal Noetherian domains.  For (1) and
  (2) this is obvious, and for (3) this is Lemma~\ref{yes}.
  Conversely, suppose $R$ is a dp-minimal integral domain.  We claim
  that $R$ falls into one of the three cases.  By
  Theorem~\ref{dft}, $R$ is a local ring.  By
  Theorem~\ref{tricho1}, one of three things happens:
  \begin{itemize}
  \item $R$ is a field.  This is case (1).
  \item $R$ is not a field; the fraction field and residue field both
    have characteristic 0.  Then the residue field is infinite, and so
    $R$ is a valuation ring by Example~\ref{deH}.
  \item $R$ is not a field; the fraction field has characteristic 0
    and the residue field is finite.  By Proposition~\ref{um}(2) and
    Lemma~\ref{where}, the integral closure $\tilde{R}$ is a
    dp-minimal DVR.  By Proposition~\ref{um}(4), % Corollary~\ref{hcol},
    $\tilde{R}$ is mixed
    characteristic.  By Lemma~\ref{moreover}, there is non-zero $a$
    such that $R \supseteq a \tilde{R}$.  Note \[ a \in a \tilde{R}
    \subseteq R \subseteq \tilde{R},\] so $a$ is a non-zero element of
    $\tilde{R}$.  Then $\tilde{R}/a\tilde{R}$ is finite by
    Lemma~\ref{finquot}, implying that $\tilde{R}/R$ is
    finite. \qedhere
  \end{itemize}
\end{proof}

\begin{acknowledgment}
  The author was supported by the National Natural Science Foundation
  of China (Grant No.\@ 12101131).  The author would like to thank two anonymous referees who found a number of errors and offered many helpful comments.
\end{acknowledgment}

%% \begin{competing}
%%   The author declares none.
%% \end{competing}

\bibliographystyle{plain} \bibliography{mybib}{}

\end{document}